\documentclass[10pt,a4paper,times]{article}
\usepackage[latin1,applemac]{inputenc}
\usepackage{amsmath, amssymb, amsfonts, amsthm, amscd, epsfig, multicol, marvosym, float,mathrsfs,enumerate,appendix,float}
\usepackage[left=2cm,top=3cm,bottom=3cm,right=2cm]{geometry}
\usepackage[scaled=.95]{helvet} 
\usepackage{courier} 

\numberwithin{equation}{section}
\newtheorem{Lemma}{Lemma}[section]
\newtheorem{Thm}{Theorem}[section]

\begin{document}
\title{Legendre Hyperelliptic integrals, $\pi$ new formulae and Lauricella functions through the elliptic singular moduli\footnote{To appear  in Journal of Numebr Theory}}
\author{Giovanni Mingari Scarpello \footnote{giovannimingari@yahoo.it}
\and Daniele Ritelli \footnote{Dipartimento scienze statitistiche, viale
Filopanti, 5 40127 Bologna Italy, daniele.ritelli@unibo.it}}
\date{}
\maketitle

\begin{abstract}
This paper, pursuing the work started in \cite{jnt1} and \cite{jnt2}, holds six new formulae for $\pi$, see equations \eqref{pi}, through ratios of first kind complete elliptic integrals and some values of  hypergeometric functions of three or four variables of Lauricella type ${\rm F}_D^{(n)}$. This will be accomplished by reducing some hyperelliptic integrals to elliptic through the methods Legendre taught in his treatise, \cite{leg, leg2}.
The complete elliptic integrals of first kind have complementary moduli: as a consequence we can find their ratio through the Lauricella ${\rm F}_D^{(3)}$ functions. In such a way we succeed in obtaining, through the theory of elliptic singular moduli, some particular values of Lauricella's themselves.

\vspace{0.5cm}

\end{abstract}

\section{Introduction and aim of the work}
In the first volume of his {\it Trait\'e} \cite{leg}, composed of 35 chapters, Legendre 
devotes more than a half of it to the elliptic integrals of first, second and third kind, showing their properties and explaining how to compute them by series, moduli transformation and so on. In chapter XXIV he starts with the {\it reduction}, namely  transforming those integrals nowadays named hyperelliptic, in somewhat else which is either elliptic or a combination of elliptic and elementary integrals. They come as a generalization of: 
\[
\int R\left(x, \sqrt{g(x)}\right)\,{\rm d}x
\]
where $R$ is a rational function, whenever the polynomial $g(x)$ has degree $\geq 5$ without multiple factors. For instance, Chapter XXVI deals with the reduction of
\[
\int\frac{{\rm d}x}{(1\mp px^2) {\sqrt[3]{1\mp qx^2}}}
\]
which, by means of ingenuous changes of variables is led to a sum of other integrals either elementary or elliptic.
Chapter XXVII reduces the integration of
\[
\int\frac{m+x^2} {\sqrt{\alpha+\beta x+\gamma x^2+\varepsilon x^4}}{\rm d}x
\]
to a sum of elliptic integrals of first and second kind. After this, he passes to the rather unaccustomed integration of
\[
\int\frac{{\rm d}\varphi}{\sqrt[m]{1-k^2\sin^2 \varphi}},\quad m=3,\, 4\quad\text{and of}\quad
\int\sqrt{\frac{1-k^2\sin^2 \varphi}{\sin\varphi}}{\rm d}\varphi
\]
Of course we cannot mention all his reductions, but we add only this last example, included in his second volume (pp. 382-398), namely
\[
\int \frac{{\rm d}x}{\sqrt{1\mp x^{m}}},\quad m=6,\, 8,\,12.
\]
Legendre in his {\it Trait\'{e}}, \cite{leg}, computes as elliptic those integrals which had been known in literature as Eulerian ones, namely computable by means of Gamma and Beta functions. Really he does so in order to get identities between complete elliptic integrals of different moduli.
We are doing somewhat similar but with reference to hyperelliptic ones: more precisely, we are interested in the reduction of six hyperelliptic definite integrals depending on parameters $a>b>0$:

\begin{subequations}\label{r}
\begin{minipage}[b]{0.45\textwidth}
\begin{align} \label{r1}
I_1(a,b)=&\int_0^\infty\frac{{\rm d}p}{\sqrt{p(p^2+a^2)(p^2+b^2)}}\\
I_2(a,b)=&\int_0^\infty\sqrt{\frac{p}{(p^2+a^2)(p^2+b^2)}}\,{\rm d}p\label{r2}\\
I_3(a,b)=&\int_a^\infty\frac{{\rm d}p}{\sqrt{p(p^2-a^2)(p^2-b^2)}}\label{r3}
\end{align}
\end{minipage}
\hfill
\begin{minipage}[b]{0.45\textwidth}
\begin{align} 
I_4(a,b)=&\int_a^\infty\sqrt{\frac{p}{(p^2-a^2)(p^2-b^2)}}\,{\rm d}p\label{r4}\\
I_5(a,b)=&\int_0^b\frac{{\rm d}p}{\sqrt{p(p^2-a^2)(p^2-b^2)}}\label{r5}\\
I_6(a,b)=&\int_0^b\sqrt{\frac{p}{(p^2-a^2)(p^2-b^2)}}\,{\rm d}p\label{r6}
\end{align}
\end{minipage}
\end{subequations}

\noindent  
 In chapter XXIX  of {\it Trait\'{e}} Legendre deals with some identities among first kind complete elliptic integrals of different moduli. Having been previously established the famous  modular transformation:
\begin{equation}\label{landen}
\mathbf{K}(k)=\frac{1}{1+k}\mathbf{K}\left(\frac{2\sqrt{k}}{1+k}\right)
\end{equation}
 where 
 \begin{equation}
\mathbf{K}(k)=\int_{0}^{1}\frac{\mathrm{d}u}{\sqrt{(1-u^{2})(1-k^{2}u^{2})}}
\end{equation}
 is the complete elliptic integral of first kind, with modulus $|k|<1$,  at p. 192 he writes the identity (which in modern notation becomes):
\begin{equation}\label{p192}
\mathbf{K}\left(\frac{\sqrt2-\sqrt[4]3}{1+\sqrt3}\right)=\frac{\sqrt2}{\sqrt[4]{27}(\sqrt3-1)}\mathbf{K}\left(\frac{1}{\sqrt2}\right).
\end{equation}
In order to get \eqref{p192}, Legendre faces by two different means the hyperelliptic integral:
\begin{equation}\label{start}
\mathbb{X}=\int_0^1\frac{x^2}{\sqrt{1-x^{12}}}\,{\rm d}x
\end{equation}
 Legendre was aware that integral \eqref{start} could be evaluated by means of Gamma function, as can be read nowadays in  formula 3.251.1 of \cite{gr}
\[
\mathbb{X}=\frac{\sqrt{\pi }}{12}\,\frac{\Gamma \left(\frac{1}{4}\right)}{\Gamma
   \left(\frac{3}{4}\right)}=\frac{\Gamma^2 \left(\frac{1}{4}\right)}{12 \sqrt{2 \pi }}
\]
but providing $\mathbb{X}$ by means of complete elliptic integrals,  through such a double path he establishes new modular relationships concerning the elliptic integrals. 
%which is not registered in any of known repertories like \cite{by} or \cite{gr}. 
For, he makes the change of variable
\begin{equation*}\label{change}
x^6=\frac{1-\cos^2\xi}{1+\cos^2\xi} %\implies\begin{cases} x=0\to \xi=0\\x=1\to \xi=\frac\pi2\\{\rm d}t=\dfrac{u^3}{\sqrt{1-u^4}}{\rm d}u\end{cases}
\end{equation*}
obtaining
\begin{equation}\label{start1}\tag{\ref{start}a}
\mathbb{X}=\frac{1}{3\sqrt2}\int_0^{\frac\pi2}\frac{{\rm d}\xi}{\sqrt{1-\frac12\sin^2\xi}}=\frac{1}{3\sqrt2}\mathbf{K}\left(\frac{1}{\sqrt2}\right).
\end{equation}
However also our change $x^3=u$ gains its ends, in fact:
\begin{equation}\label{start2}\tag{\ref{start}b}
\mathbb{X}=\frac{1}{3}\int_0^1\frac{{\rm d}u}{\sqrt{1-u^4}}.
\end{equation}
Furthermore Legendre, through the change $1-x^4=p x^2$, gets:
\begin{equation}\label{start22}\tag{\ref{start}c}
\mathbb{X}=\frac12\int_0^\infty\frac{{\rm d}p}{\sqrt{p(p^2+3)(p^2+4)}}.
\end{equation}
The integral in \eqref{start22} is a special case of the family at right hand side of \eqref{r1}. Legendre then considers a special case of $I_3(a,b)$ at p. 193 of \cite{leg}. Having previously stated the identity
\[
\int_0^\infty\frac{z^2}{\sqrt{1+z^{12}}}{\rm d}z=\frac12\int_2^\infty\frac{{\rm d}p}{\sqrt{p(p^2-3)(p^2-4)}}
\]
 he obtains another modular transformation for complete first kind elliptic integrals
\begin{equation}\label{p192b}
\mathbf{K}\left(\frac{\sqrt2+\sqrt[4]3}{1+\sqrt3}\right)=\sqrt{3+2\sqrt3}\,\mathbf{K}\left(\frac{1}{\sqrt2}\right).
\end{equation}
In \cite{leg2} chapter III section 23 p. 384, in order to reduce an integral like $I_4(a,b),$ he gets the identity
\[
\int_0^\infty\frac{1+z^4}{\sqrt{1+z^{12}}}{\rm d}z=\frac12\int_2^\infty\sqrt{\frac{p}{(p^2-3)(p^2-4)}}\,{\rm d}p.
\]
Our study for reducing \eqref{r} enlarges the Legendre approach to whichever value of parameters $a,\,b$ and to integrals  $I_2(a,b),\,I_5(a,b),\,I_6(a,b)$, whose reduction, as far as we are concerned, is presently not available in literature.

\section{Reduction of integrals \eqref{r} to elliptic ones}

For, let us start with the change of variable $p>0$
\begin{equation}\label{cambio}
f(p)=:p+\frac{ab}{p}=u\iff p=\frac{1}{2}\left(u\mp\sqrt{u^2-4ab}\right)
\end{equation}
It should be observed that, such a transformation used by Legendre, nowadays is called as {\it Cauchy-Schl{\"o}milch transformation} and it is a powerful tool in order to evaluate definite integrals. Recently in \cite{Schlo}, where some hystorical references are given, many integrals are evaluated by this transformation and among them there is a couple of elliptic integrals. 
The function describing it, whose plot is shown in Figure 1, defined in \eqref{cambio}, cannot be inverted: $f(p)\to\infty$ as $p\to0^{+}$ and as $p\to\infty$ so that the point $(\sqrt{ab},2\sqrt{ab})$ is an absolute minimum.
\begin{figure}[H]
\begin{center}
\scalebox{0.75}{\includegraphics{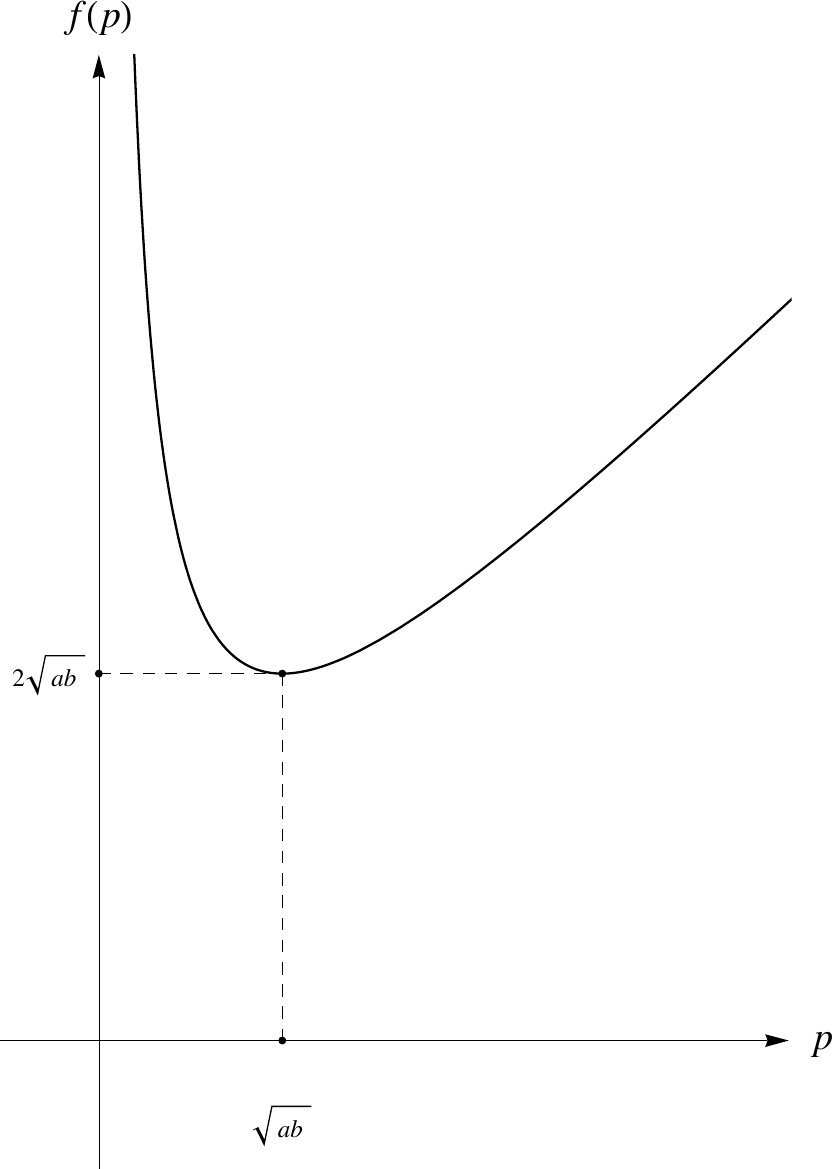}}\label{xf01} 
\end{center}
\caption{Cauchy-Schl{\"o}milch transformation}
\end{figure}
Thus, in order to handle the change of variable \eqref{cambio}, we have to split the relevant integration domain in two sub-intervals $[0,\sqrt{ab}]$ and $[\sqrt{ab},\infty).$ Let us treat apart each case in detail.

\subsection*{$\boldsymbol{p\in[\sqrt{ab},\infty)}$}
In such a case \eqref{cambio} has to been taken with the sign $+$ and then, by means of a formula on nested radicals we have:

\begin{subequations}\label{cambiopiu}
\begin{minipage}[b]{0.45\textwidth}
\begin{align}
&{\rm d}p=\frac{\sqrt{u^2-4ab}+u}{2\sqrt{u^2-4ab}}\,{\rm d}u\label{cambiopiua}
\end{align}
\end{minipage}
\hfill
\begin{minipage}[b]{0.45\textwidth}
\begin{align}
&\sqrt{p}=\frac{1}{2}\left(\sqrt{u+2\sqrt{ab}}+\sqrt{u-2\sqrt{ab}}\right).\label{cambiopiub}
\end{align}
\end{minipage}
\end{subequations}

\noindent Now we evaluate the denominators of the six integrands in \eqref{r}
\begin{subequations}\label{denompiu}
\begin{align}
(p^2+a^2)(p^2+b^2)&=\frac{1}{2}\left[u^2+(a-b)^2\right] \left(u^2-2 a b+u\sqrt{u^2-4 a b}\right)\label{piupiupiu}\\
(p^2-a^2)(p^2-b^2)&=\frac{1}{2}\left[u^2-(a+b)^2\right] \left(u^2-2 a b+u\sqrt{u^2-4 a b}\right)\label{piumenmen}
\end{align}
\end{subequations}
\noindent and again by nested radicals relationship we have:
\begin{subequations}\label{piudoppi}
\begin{align}
\sqrt{(p^2+a^2)(p^2+b^2)}&=\frac{1}{2}\left(u+\sqrt{u^2-4 a b}\right)\sqrt{u^2+(a-b)^2}\label{piudoppipiu}\\[2mm]
\sqrt{(p^2-a^2)(p^2-b^2)}&=\frac{1}{2}\left(u+\sqrt{u^2-4 a b}\right)\sqrt{u^2-(a+b)^2} \label{piudoppimen}
\end{align}
\end{subequations}

\subsection*{$\boldsymbol{p\in[0,\sqrt{ab}]}$}
In such a case \eqref{cambio} has to be taken with the sign $-$ so that we find

\begin{subequations}\label{cambiomen}
\begin{minipage}[b]{0.45\textwidth}
\begin{align}
&{\rm d}p=\frac{\sqrt{u^2-4ab}-u}{2\sqrt{u^2-4ab}}\,{\rm d}u\label{cambiomena}
\end{align}
\end{minipage}
\hfill
\begin{minipage}[b]{0.45\textwidth}
\begin{align}
&\sqrt{p}=\frac{1}{2}\left(\sqrt{u+2\sqrt{ab}}-\sqrt{u-2\sqrt{ab}}\right)\label{cambiomenb}
\end{align}
\end{minipage}
\end{subequations}

\noindent Going ahead to evaluate all the denominators of the integrands of \eqref{r}
\begin{subequations}\label{denommen}
\begin{align}
(p^2+a^2)(p^2+b^2)&=\frac{1}{2}\left[u^2+(a-b)^2\right] \left(u^2-2 a b-u\sqrt{u^2-4 a b}\right)\label{menpiupiu}\\
(p^2-a^2)(p^2-b^2)&=\frac{1}{2}\left[u^2-(a+b)^2\right] \left(u^2-2 a b-u\sqrt{u^2-4 a b}\right)\label{menmenmen}
\end{align}
\end{subequations}
then we arrive at:
\begin{subequations}\label{mendoppi}
\begin{align}
\sqrt{(p^2+a^2)(p^2+b^2)}&=\frac{1}{2}\left(u-\sqrt{u^2-4 a b}\right)\sqrt{u^2+(a-b)^2}\label{mendoppipiu}\\[2mm]
\sqrt{(p^2-a^2)(p^2-b^2)}&=\frac{1}{2}\left(u-\sqrt{u^2-4 a b}\right)\sqrt{u^2-(a+b)^2} \label{mendoppimen}
\end{align}
\end{subequations}
By means of the above changes we can get several reduction formulae.
\begin{Lemma}
If $a>b>0$ then
\begin{subequations}\label{riduzione1}
\begin{align}
I_1(a,b)&=\frac{1}{\sqrt{ab}}\int_{2\sqrt{ab}}^\infty\frac{{\rm d}u}{\sqrt{(u-2\sqrt{ab})[u^2+(a-b)^2]}}\label{I1}\\
I_2(a,b)&=\int_{2\sqrt{ab}}^\infty\frac{{\rm d}u}{\sqrt{(u-2\sqrt{ab})[u^2+(a-b)^2]}}\label{I2}\\
I_3(a,b)&=\frac{1}{2\sqrt{ab}}\int_{a+b}^\infty\frac{1}{\sqrt{u^2-(a+b)^2}}\left(\frac{1}{\sqrt{u-2\sqrt{ab}}}-\frac{1}{\sqrt{u+2\sqrt{ab}}}\right){\rm d}u\label{I3}\\
I_4(a,b)&=\frac{1}{2}\int_{a+b}^\infty\frac{1}{\sqrt{u^2-(a+b)^2}}\left(\frac{1}{\sqrt{u-2\sqrt{ab}}}+\frac{1}{\sqrt{u+2\sqrt{ab}}}\right){\rm d}u\label{I4}\\
I_5(a,b)&=\frac{1}{2\sqrt{ab}}\int_{a+b}^\infty\frac{1}{\sqrt{u^2-(a+b)^2}}\left(\frac{1}{\sqrt{u-2\sqrt{ab}}}+\frac{1}{\sqrt{u+2\sqrt{ab}}}\right){\rm d}u\label{I5}
\end{align}
\begin{align}
I_6(a,b)&=\frac{1}{2}\int_{a+b}^\infty\frac{1}{\sqrt{u^2-(a+b)^2}}\left(\frac{1}{\sqrt{u-2\sqrt{ab}}}-\frac{1}{\sqrt{u+2\sqrt{ab}}}\right){\rm d}u\label{I6}
\end{align}
\end{subequations}
\end{Lemma}
\begin{proof}
Let us split the integration interval as $[0,\sqrt{ab}]\cup(\sqrt{ab},\infty).$ In the first one we make use of \eqref{cambiopiu} and \eqref{piupiupiu}
 \begin{equation}\label{sub1}
 \begin{split}
 \int_0^{\sqrt{ab}}\frac{{\rm d}p}{\sqrt{p(p^2+a^2)(p^2+b^2)}}&=\frac{1}{2\sqrt{ab}}\int_{2\sqrt{ab}}^\infty\frac{{\rm d}u}{\sqrt{(u+2\sqrt{ab})[u^2+(a-b)^2]}}\\
 &+\frac{1}{2\sqrt{ab}}\int_{2\sqrt{ab}}^\infty\frac{{\rm d}u}{\sqrt{(u-2\sqrt{ab})[u^2+(a-b)^2]}}
 \end{split}
 \end{equation}
Analogously in $[\sqrt{ab},\infty)$ by \eqref{cambiomen} and \eqref{menpiupiu} one finds
 \begin{equation}\label{sub2}
 \begin{split}
 \int_{\sqrt{ab}}^\infty\frac{{\rm d}p}{\sqrt{p(p^2+a^2)(p^2+b^2)}}&=\frac{1}{2\sqrt{ab}}\int_{2\sqrt{ab}}^\infty\frac{{\rm d}u}{\sqrt{(u-2\sqrt{ab})[u^2+(a-b)^2]}}\\
 &-\frac{1}{2\sqrt{ab}}\int_{2\sqrt{ab}}^\infty\frac{{\rm d}u}{\sqrt{(u+2\sqrt{ab})[u^2+(a-b)^2]}}
 \end{split}
 \end{equation}
Thesis \eqref{I1} follows joining \eqref{sub1} with \eqref{sub2}, and through the same way one will be driven to \eqref{I2}. In order to prove \eqref{I3} and \eqref{I4}, we will act along the increasing  branch of the function \eqref{cambio} invoking \eqref{cambiopiu} and \eqref{piudoppipiu}. On the contrary for \eqref{I5} and \eqref{I6}, one shall act along the decreasing  branch of \eqref{cambio} invoking \eqref{cambiomen} and \eqref{mendoppimen}.

\end{proof}
Now we are ready to compute the complete elliptic integrals at right side of \eqref{riduzione1}. For shortening we put:
\begin{equation}\label{mentari}
\mathbf{K}_{+}^{(a,b)}=\mathbf{K}\left(\frac{\sqrt{a}+\sqrt{b}}{\sqrt{2(a+b)}}\right),\qquad \mathbf{K}_{-}^{(a,b)}=\mathbf{K}\left(\frac{\sqrt{a}-\sqrt{b}}{\sqrt{2(a+b)}}\right).
\end{equation}

\begin{Thm}\label{teocompl}
The following formulae hold, $a>b>0$

\begin{subequations}\label{compl}
\begin{minipage}[b]{0.45\textwidth}
\begin{align} 
I_1(a,b)&=\frac{2}{\sqrt{ab(a+b)}}\,\mathbf{K}_{-}^{(a,b)}\label{compl1}\\
I_2(a,b)&=\frac{2}{\sqrt{a+b}}\,\mathbf{K}_{-}^{(a,b)}\label{compl2}\\
I_3(a,b)&=\frac{1}{\sqrt{2ab(a+b)}}\left[\mathbf{K}_{+}^{(a,b)}-\mathbf{K}_{-}^{(a,b)}\right]\label{compl3}
\end{align}
\end{minipage}
\hfill
\begin{minipage}[b]{0.45\textwidth}
\begin{align} 
I_4(a,b)&=\frac{1}{\sqrt{2(a+b)}}\left[\mathbf{K}_{+}^{(a,b)}+\mathbf{K}_{-}^{(a,b)}\right]\label{compl4}\\
I_5(a,b)&=\frac{1}{\sqrt{2ab(a+b)}}\left[\mathbf{K}_{+}^{(a,b)}+\mathbf{K}_{-}^{(a,b)}\right]\label{compl5}\\
I_6(a,b)&=\frac{1}{\sqrt{2(a+b)}}\left[\mathbf{K}_{+}^{(a,b)}-\mathbf{K}_{-}^{(a,b)}\right]\label{compl6}
\end{align}
\end{minipage}
\end{subequations}

\end{Thm}
\begin{proof}
For $I_1(a,b)$ e $I_2(a,b)$, let us start from \eqref{I1} and \eqref{I2}. By \cite{gr}, entry 3.138-7 p. 259 or \cite{by} entry 239.00 p. 86, we have:\[
\int_\alpha^\infty\frac{{\rm d}x}{\sqrt{(x-\alpha)(x^2+\nu^2)}}=\frac{2}{\sqrt[4]{\alpha^2+\nu^2}}\mathbf{K}\left(\sqrt{\frac{\sqrt{\alpha^2+\nu^2}-\alpha}{2\sqrt{\alpha^2+\nu^2}}}\right)
\]
For the elliptic integrals concerning $I_3(a,b),\dots,I_6(a,b)$ we will use entry 238.00 p. 84 of \cite{by}: if $\alpha>\beta>\gamma:$
\[
\int_\alpha^\infty\frac{{{\rm d}x}}{\sqrt{(x-\alpha)(x-\beta)(x-\gamma)}}=\frac{2}{\sqrt{\alpha-\gamma}}\mathbf{K}\left(\sqrt{\frac{\beta-\gamma}{\alpha-\gamma}}\right)
\]
\end{proof}

\section{Some new $\boldsymbol{\pi}$ formulae through hypergeometric integrals}
Integrals $I_1(a,b),\dots,I_6(a,b)$ can be evaluated by means of Lauricella multi-variables hypergeometric functions. 
We assumed a reader's previous knowledge on the elliptic functions (see \cite{by} for
details), but for his better understanding we provide here a minimal account on the Lauricella functions' huge subject.
\subsection{The Lauricella functions' tool}

It is well known that hypergeometric series first appeared in the Wallis's \textit{Arithmetica
infinitorum} (1656): 
\[
_{2}\mathrm{F}_{1}\left( \left. 
\begin{array}{c}
a;b \\[2mm]
c
\end{array}
\right| x\right)=1+\frac{a\cdot b}{1\cdot c}x+\frac{a\cdot (a+1)\cdot b%
\cdot (b+1)}{1\cdot 2\cdot c\cdot (c+1)}x^{2}+\cdots ,
\]
for $|x|<1$ and real parameters $a,\,b,\,c.$ The product of $n$ factors: 
\[
(\lambda )_{n}=\lambda \left( \lambda +1\right) \cdots \left( \lambda
+n-1\right)=\frac{\Gamma (\lambda+n)}{\Gamma (\lambda)} ,
\]
called \textit{Pochhammer symbol} (or \textit{truncated factorial})
where $ \Gamma (\cdot)$ is the Eulerian integral of second kind, allows to write $_{2}{\rm F}_{1}$ as: 
\[
_{2}\mathrm{F}_{1}\left( \left. 
\begin{array}{c}
a;b \\[2mm]
c
\end{array}
\right| x\right)=\sum_{n=0}^{\infty }\frac{\left( a\right) _{n}\left(
b\right) _{n}}{\left( c\right) _{n}}\frac{x^{n}}{n!}.
\]
The first meaningful contributions about various $_{2}{\rm F}_{1}$ representations
are due to Euler\footnote{%
Our historical references are, beyond Wallis, the following Euler works: a) \textit{De progressionibus transcendentibus}, Op.
omnia, S.1, vol. 28; b) \textit{De curva hypergeometrica} Op. omnia, S.1,
vol. 16; c) \textit{Institutiones Calculi integralis} , 1769, vol.II.};
nevertheless he doesn't seem to have known the $ _{2}{\rm F}_{1} $ integral representation: 
\[
_{2}\mathrm{F}_{1}\left( \left. 
\begin{array}{c}
a;b \\[2mm]
c
\end{array}
\right| x\right) =\frac{\Gamma (c)}{\Gamma (c-a)\Gamma (a)}\int_{0}^{1}%
\frac{u^{a-1}(1-u)^{c-a-1}}{(1-x\,u)^{b}}\,\mathrm{d}u
\]
traditionally ascribed to him, but really due to A. M. Legendre\footnote{%
A. M. Legendre, \textit{Exercices de calcul int\'{e}gral}, II, quatri\'{e}me
part, sect. 2, Paris 1811}. For all this and for the Stirling contributions,
see \cite{dutka}. The above integral relationship is true if $c>a>0$ and for 
$\left| x\right| <1,$ even if this bound can be discarded due to analytic
continuation.\ 

Many functions have been introduced in 19$^{\mathrm{th}}$ century for
generalizing the hypergeometric functions to multiple variables, but we will
mention only those introduced and investigated by G. Lauricella (1893) and
S. Saran (1954).\ Among them our interest is focused on the hypergeometric
function ${\rm F}_{D}^{(n)}$ of $n\in \mathbb{N}^{+}$ variables (and $n+2$
parameters), see \cite{saran1954} and \cite{lauricella1893}, defined as: 
\[
\mathrm{F}_{D}^{(n)}\left( \left. 
\begin{array}{c}
a;b_{1},\ldots ,b_{n} \\[2mm]
c
\end{array}
\right| x_{1},\ldots ,x_{n}\right):=
\sum_{m_{1},\ldots ,m_{n}\in \mathbb{N}}\frac{(a)_{m_{1}+\cdots
+m_{n}}(b_{1})_{m_{1}}\cdots (b_{n})_{m_{n}}}{(c)_{m_{1}+\cdots
+m_{n}}m_{1}!\cdots m_{n}!}\,x_{1}^{m_{1}}\cdots x_{n}^{m_{m}} 
\]
with the hypergeometric series usual convergence requirements $%
|x_{1}|<1,\ldots ,|x_{n}|<1$. If $\mathrm{Re}\,c>\mathrm{Re}\,a>0$, the
relevant Integral Representation Theorem $\left( \text{IRT}\right) $
provides: 
\begin{equation}\label{irto}
\mathrm{F}_{D}^{(n)}\left( \left. 
\begin{array}{c}
a;b_{1},\ldots ,b_{n} \\[2mm]
c
\end{array}
\right| x_{1},\ldots ,x_{n}\right) =\frac{%
\Gamma (c)}{\Gamma (a)\,\Gamma (c-a)}\,\int_{0}^{1}\,\frac{%
u^{a-1}(1-u)^{c-a-1}}{(1-x_{1}u)^{b_{1}}\cdots (1-x_{n}u)^{b_{n}}}\,\mathrm{d%
}u 
\end{equation}
allowing the analytic continuation to $\mathbb{C}^{n}$ deprived of the
cartesian $n$-dimensional product of the interval $]1,\infty [$ with itself. Finally, we quote the reduction formula for Lauricella functions:
\begin{equation} \label{id:3}
\mathrm{F}_{D}^{(n)}\left( \left. 
\begin{array}{c}
a;b_{1},\ldots ,b_{n} \\[2mm]
b_{1}+\cdots +b_{n}
\end{array}
\right| x_{1},\ldots ,x_{n}\right) =\frac{1}{(1-x_{n})^{a}}\,\mathrm{F}%
_{D}^{(n-1)}\left( \left. 
\begin{array}{c}
a,b_{2},\ldots ,b_{n} \\ 
b_{1}+\cdots +b_{n}
\end{array}
\right| \frac{x_{1}-x_{n}}{1-x_{n}},\ldots ,\frac{x_{n-1}-x_{n}}{1-x_{n}}%
\right)  
\end{equation}
whose proof is given in \cite{jnt2}.

It should be highlighted that, after having been introduced as a generalization within the hypergeometric framework of those of  Gauss and Appell ones, the Lauricella functions have started to be used, see \cite{exx}, in non linear Mathematical Physics and Statistics,  e.g. in order to treat some hyperelliptic integrals: see \cite{Brac} or \cite{Elas}. Moreover the cosmologist Kraniotis \cite{kraniotis2011precise, kraniotis2005frame} recurred to Lauricella functions in order to treat Kerr and Kerr-(anti) de Sitter black holes, and the bending of light in Kerr and Kerr-(anti) de Sitter spacetimes. Nevertheless their employ presently is still not so popular.

In such a way, pursuing our research path started in \cite{jnt1} and gone on with \cite{jnt2}, we will obtain new identities linking some of the Lauricella functions to complete elliptic integrals and some evaluation in their analytic continuation, theorem \ref{prolo}. Then in section 4 we will provide new values, even if in very special cases, of Lauricella functions. 

\subsection{Lauricella functions in action}
We begin with a small notation detail in order to simplify our writing:  whenever  in a Lauricella function the $n$ parameters type $b$ are all equal, $b_{1}=\cdots=b_{n}=b$, we put
\[
\mathrm{F}_{D}^{(n)}\left( \left. 
\begin{array}{c}
a;b_{1},\ldots ,b_{n} \\[2mm]
c
\end{array}
\right| x_{1},\ldots ,x_{n}\right) =\mathrm{F}_{D}^{(n)}\left( \left. 
\begin{array}{c}
a;b \\[2mm]
c
\end{array}
\right| x_{1},\ldots ,x_{n}\right)
\]
the repetitions number of $b$ is however set by the apex of the  Lauricella function.

Up to this point, we can feel able to compute the integrals $I_1(a,b),\dots,I_6(a,b)$ through the above mentioned hypergeometric integral representation theorem.
\begin{Lemma}\label{lemmalau}
If $a>b>0$ then
\begin{subequations}\label{lauri1}
\begin{align}
I_1(a,b)&=\frac{\pi}{2}\,\mathrm{F}_{D}^{(4)}\left( \left. 
\begin{array}{c}
\frac32;\frac12 \\[2mm]
2
\end{array}
\right| 1+ia,1-ia,1+ib,1-ib\right)\label{LL1}\\
I_2(a,b)&=\frac{\pi}{2}\,\mathrm{F}_{D}^{(4)}\left( \left. 
\begin{array}{c}
\frac12;\frac12 \\[2mm]
2
\end{array}
\right| 1+ia,1-ia,1+ib,1-ib\right)\label{LL2}\\
I_3(a,b)&=\frac{\pi}{2\sqrt{2a(a^2-b^2)}}\mathrm{F}_{D}^{(3)}\left( \left. 
\begin{array}{c}
\frac12;\frac12 \\[2mm]
2
\end{array}
\right| \frac12,\frac{b}{a+b},\frac{b}{b-a}\right)\label{LL3}\\
I_4(a,b)&=\pi\sqrt{\frac{a}{2(a^2-b^2)}}\mathrm{F}_{D}^{(3)}\left( \left. 
\begin{array}{c}
\frac12;\frac12 \\[2mm]
1
\end{array}
\right| \frac12,\frac{b}{a+b},\frac{b}{b-a}\right)\label{LL4}\\
I_5(a,b)&=\frac{\pi}{a\sqrt{b}}\mathrm{F}_{D}^{(3)}\left( \left. 
\begin{array}{c}
\frac12;\frac12 \\[2mm]
1
\end{array}
\right| -1,\frac{b}{a},-\frac{b}{a}\right)\label{LL5}\\
I_6(a,b)&=\frac{\pi\sqrt{b}}{2a}\mathrm{F}_{D}^{(3)}\left( \left. 
\begin{array}{c}
\frac32;\frac12 \\[2mm]
2
\end{array}
\right| -1,\frac{b}{a},-\frac{b}{a}\right)\label{LL6}
\end{align}
\end{subequations}
\end{Lemma}
\begin{proof}
In $I_1(a,b)$ let us do $p=(1-t)/t$ obtaining
\begin{equation*}\label{iper}
\int_0^1\frac{t^{1/2}(1-t)^{-1/2}}{\sqrt{\left[\left(1+a^2\right) t^2-2 t+1\right]\left[\left(1+b^2\right) t^2-2 t+1\right]}}\,{\rm d}t.
\end{equation*}
Thesis \eqref{LL1} stems by the integral representation theorem \eqref{irto}, if one observes that: 
\[
%\begin{split}
\left[\left(1+a^2\right) t^2-2 t+1\right]\left[\left(1+b^2\right) t^2-2 t+1\right]=\left[1-(1-ia)t\right]\left[1-(1+ia)t\right]\left[1-(1-ib)t\right]\left[1-(1+ib)t\right]
%\end{split}
\]
The proof of \eqref{LL2} can be achieved likewise. Next, \eqref{LL3} and \eqref{LL4} can be proved by means of the change of variable $t=(p-a)/p$ which leads to
\[
\begin{split}
I_3(a,b)&=\frac{1}{\sqrt{a}}\int_0^1\sqrt{\frac{1-t}{t(2-t)[a^2-b^2(1-t^2)]}}\,{\rm d}t\\
I_4(a,b)&=\sqrt{a}\int_0^1\frac{1}{\sqrt{t(1-t)(2-t)[a^2-b^2(1-t)^2]}}\,{\rm d}t
\end{split}
\]
and hereinafter to \eqref{LL3} and \eqref{LL4} by means of the representation theorem \eqref{irto}. 
Finally, for \eqref{LL5} and \eqref{LL6} the change of variable $p=bt$ leads to:
\[
\begin{split}
I_5(a,b)&=\frac{1}{\sqrt{b}}\int_0^1\frac{{\rm d}t}{\sqrt{t(1-t^2)(a^2-b^2t^2)}}\\
I_6(a,b)&=\sqrt{b}\int_0^1\sqrt{\frac{t}{(1-t^2)(a^2-b^2t^2)}}\,{\rm d}t
\end{split}
\]
then apply again the representation theorem.
\end{proof}
Comparing theses of lemma \ref{lemmalau} and of  theorem \ref{teocompl} we obtain immediately our six new formulae for $\pi$ to be added to those of our previous articles \cite{jnt1} and \cite{jnt2}.\begin{Thm}\label{newpi}
For $a>b>0$ the following formulae hold:  
\begin{subequations}\label{pi}
\begin{align} 
\pi=&\frac{4}{\sqrt{ab(a+b)}}\,\frac{\mathbf{K}_{-}^{(a,b)}}{\mathrm{F}_{D}^{(4)}\left( \left. 
\begin{array}{c}
\frac32;\frac12 \\[2mm]
2
\end{array}
\right| 1+ia,1-ia,1+ib,1-ib\right)}\label{p1}\\
\pi=&\frac{4}{\sqrt{a+b}}\,\frac{\mathbf{K}_{-}^{(a,b)}}{\mathrm{F}_{D}^{(4)}\left( \left. 
\begin{array}{c}
\frac12;\frac12 \\[2mm]
2
\end{array}
\right| 1+ia,1-ia,1+ib,1-ib\right)}\label{p2}\\
\pi=&\frac{2\sqrt{a-b}}{\sqrt{b}}\frac{\mathbf{K}_{+}^{(a,b)}-\mathbf{K}_{-}^{(a,b)}}{\mathrm{F}_{D}^{(3)}\left( \left. 
\begin{array}{c}
\frac12;\frac12 \\[2mm]
2
\end{array}
\right| \frac12,\frac{b}{a+b},\frac{b}{b-a}\right)}\label{p3}\\
\pi=&\frac{\sqrt{a-b}}{\sqrt{a}}\frac{\mathbf{K}_{+}^{(a,b)}+\mathbf{K}_{-}^{(a,b)}}{\mathrm{F}_{D}^{(3)}\left( \left. 
\begin{array}{c}
\frac12;\frac12 \\[2mm]
1
\end{array}
\right| \frac12,\frac{b}{a+b},\frac{b}{b-a}\right)}\label{p4}\\
\pi=&\frac{\sqrt{a}}{\sqrt{2(a+b)}}\frac{\mathbf{K}_{+}^{(a,b)}+\mathbf{K}_{-}^{(a,b)}}{\mathrm{F}_{D}^{(3)}\left( \left. 
\begin{array}{c}
\frac12;\frac12 \\[2mm]
1
\end{array}
\right| -1,\frac{b}{a},-\frac{b}{a}\right)}\label{p5}\\
\pi=&\frac{\sqrt{2} a}{\sqrt{b (a+b)}}\frac{\mathbf{K}_{+}^{(a,b)}-\mathbf{K}_{-}^{(a,b)}}{\mathrm{F}_{D}^{(3)}\left( \left. 
\begin{array}{c}
\frac32;\frac12 \\[2mm]
2
\end{array}
\right| -1,\frac{b}{a},-\frac{b}{a}\right)}\label{p6}
\end{align}
 \end{subequations}

\end{Thm}
By \eqref{LL1} and \eqref{LL2} of lemma \ref{lemmalau}, and by \eqref{p1} and \eqref{p2} of theorem \ref{teocompl} and by reduction formula for Lauricella functions \eqref{id:3} we can get two new evaluations of  $F_D^{(3)}$ in its analytic continuation with two arguments greater than 1, quite similar to the evaluation of $_2{\rm F}_1$ with argument 2 established in theorem 3.2 given in \cite{jnt2}.

\begin{Thm}\label{prolo}
If $a>b>0$ then
\begin{subequations}\label{continuation}
\begin{align}
\mathrm{F}_{D}^{(3)}\left( \left. 
\begin{array}{c}
\frac32;\frac12 \\[2mm]
2
\end{array}
\right|\frac{a+b}{b},\frac{b-a}{b},2\right)&=\frac{4(-1+i)}{\pi\sqrt2}\sqrt{\frac{b}{a+b}}\mathbf{K}_{-}^{(a,b)}\label{conti1}\\
\mathrm{F}_{D}^{(3)}\left( \left. 
\begin{array}{c}
\frac12;\frac12 \\[2mm]
2
\end{array}
\right|\frac{a+b}{b},\frac{b-a}{b},2\right)&=\frac{4(1+i)}{\pi\sqrt2}\sqrt{\frac{b}{a(a+b)}}\mathbf{K}_{-}^{(a,b)}\label{conti2}
\end{align}
\end{subequations}
\end{Thm}
                 
\section{Connections with the Elliptic Singular Moduli}
The singular value theory stems from the equation
\begin{equation}\label{sin}
\frac{\mathbf{K}'(k)}{\mathbf{K}(k)}=\sqrt{n}
\end{equation}
where: $n$ is a given positive integer, $\mathbf{K}(k)$ is the complete elliptic integral of first kind, $k'=\sqrt{1-k^2}$ and $\mathbf{K}'(k)=\mathbf{K}(\sqrt{1-k^2})$. Solution $\lambda^\ast(n)$ to \eqref{sin} is an algebraic number which is called {\it$n$-th singular value for the modulus $k$}. Function $\lambda^\ast(n)$ is called elliptic lambda function. Selberg and Chowla \cite{SC} proved the unproved statement of Ramanujan, \cite{R}, that for $n\in\mathbb{N},\, \mathbf{K}(\lambda^\ast(n))$ is expressibile on terms of $\Gamma$ function. For low $n$ values there are explicit results, cases $n=1,\,3,\,4$ treated in Wittaker and Watson \cite{WW}; Borwein and Borwein \cite{BB} gave $n=1,\,2,\,3.$ Moreover $n=2$ is treated by Glasser and Wood \cite{GW}. Selberg and Chowla \cite{SC}, Borwein and Zucker \cite{BZ}, Zucker \cite{Z}, expressed $\mathbf{K}(\lambda^\ast(n))$ in terms of $\Gamma$ for many values of $n$ using techniques based up lattices sums introduced by Zucker and Robertson \cite{ZR1, ZR2}. We recall, even if out of our applications, that if in \eqref{sin} $n\in\mathbb{R}$, the singular modulus can be given in terms of Jacobi Theta functions of zero argument, see \cite{BZ}
\[
k(n)=\left[\frac{\theta_2(0,q_n)}{\theta_3(0,q_n)}\right]^2
\]
where the nome is $q_n=e^{-\pi/\sqrt{n}}$ and 
where the theta functions are defined as:
\[
\theta_2(0,q)=\sum_{m=-\infty}^\infty q^{(m+\frac12)^2},\quad \theta_3(0,q)=\sum_{m=-\infty}^\infty q^{m^{2}}.
\]
The singular moduli theory is linked to our subject because the complete elliptic integrals $\mathbf{K}_{-}^{(a,b)}$ and $\mathbf{K}_{+}^{(a,b)}$ introduced in \eqref{mentari} are complementary and their values are quickly obtained by a system of equations \eqref{p3} and \eqref{p4}, obtaining
\begin{subequations}
\begin{align}
\mathbf{K}_{+}^{(a,b)}&=\frac{\pi}{2\sqrt{a-b}}\left[\frac{\sqrt{b}}{2}\,{\mathrm{F}_{D}^{(3)}\left( \left. 
\begin{array}{c}
\frac12;\frac12 \\[2mm]
2
\end{array}
\right| \frac12,\frac{b}{a+b},\frac{b}{b-a}\right)}+\sqrt{a}\,{\mathrm{F}_{D}^{(3)}\left( \left. 
\begin{array}{c}
\frac12;\frac12 \\[2mm]
1
\end{array}
\right| \frac12,\frac{b}{a+b},\frac{b}{b-a}\right)}\right]\label{kpiu1}\\
\mathbf{K}_{-}^{(a,b)}&=\frac{\pi}{2\sqrt{a-b}}\left[\sqrt{a}\,{\mathrm{F}_{D}^{(3)}\left( \left. 
\begin{array}{c}
\frac12;\frac12 \\[2mm]
1
\end{array}
\right| \frac12,\frac{b}{a+b},\frac{b}{b-a}\right)}-\frac{\sqrt{b}}{2}\,{\mathrm{F}_{D}^{(3)}\left( \left. 
\begin{array}{c}
\frac12;\frac12 \\[2mm]
2
\end{array}
\right| \frac12,\frac{b}{a+b},\frac{b}{b-a}\right)}\right]\label{kpiu2}
\end{align}
\end{subequations}
In such a way, having computed the ratio between \eqref{kpiu1} and \eqref{kpiu2} we get:
\begin{equation}\label{quozi}
\frac{\mathbf{K}_{+}^{(a,b)}}{\mathbf{K}_{-}^{(a,b)}}=\frac{\frac{\sqrt{b}}{2}\,{\mathrm{F}_{D}^{(3)}\left( \left. 
\begin{array}{c}
\frac12;\frac12 \\[2mm]
2
\end{array}
\right| \frac12,\frac{b}{a+b},\frac{b}{b-a}\right)}+\sqrt{a}\,{\mathrm{F}_{D}^{(3)}\left( \left. 
\begin{array}{c}
\frac12;\frac12 \\[2mm]
1
\end{array}
\right| \frac12,\frac{b}{a+b},\frac{b}{b-a}\right)}}{\sqrt{a}\,{\mathrm{F}_{D}^{(3)}\left( \left. 
\begin{array}{c}
\frac12;\frac12 \\[2mm]
1
\end{array}
\right| \frac12,\frac{b}{a+b},\frac{b}{b-a}\right)}-\frac{\sqrt{b}}{2}\,{\mathrm{F}_{D}^{(3)}\left( \left. 
\begin{array}{c}
\frac12;\frac12 \\[2mm]
2
\end{array}
\right| \frac12,\frac{b}{a+b},\frac{b}{b-a}\right)}}.
\end{equation}
In similar way by a system of \eqref{p5} and \eqref{p6} we get:
\begin{equation}\label{quozi2}
\frac{\mathbf{K}_{+}^{(a,b)}}{\mathbf{K}_{-}^{(a,b)}}=\frac{2\sqrt{a(a+b)}\,{\mathrm{F}_{D}^{(3)}\left( \left. 
\begin{array}{c}
\frac12;\frac12 \\[2mm]
1
\end{array}
\right| -1,\frac{b}{a},-\frac{b}{a}\right)}+\sqrt{b(a+b)}\,{\mathrm{F}_{D}^{(3)}\left( \left. 
\begin{array}{c}
\frac32;\frac12 \\[2mm]
2
\end{array}
\right| -1,\frac{b}{a},-\frac{b}{a}\right)}}{2\sqrt{a(a+b)}\,{\mathrm{F}_{D}^{(3)}\left( \left. 
\begin{array}{c}
\frac12;\frac12 \\[2mm]
1
\end{array}
\right| -1,\frac{b}{a},-\frac{b}{a}\right)}-\sqrt{b(a+b)}\,{\mathrm{F}_{D}^{(3)}\left( \left. 
\begin{array}{c}
\frac32;\frac12 \\[2mm]
2
\end{array}
\right| -1,\frac{b}{a},-\frac{b}{a}\right)}}.
\end{equation}
On the other side, some singular moduli have been tabulated, see \cite{BB} and \cite{wei}, then by choosing properly the parameters $a$ and $b$, the left hand side of \eqref{quozi} and \eqref{quozi2} is an algebraic radical. In such a way we can provide new relationships met by the Lauricella's, having fixed the parameters for the chosen arguments. Let us show those of them leading to simpler or, better, to not too intricate, results. Singular modulus of order  3: equating to the modulus of $\mathbf{K}_{-}^{(a,b)}$ we get the equation
\[
\lambda^\ast(3)=\frac{\sqrt{3}-1}{2 \sqrt{2}}=\frac{\sqrt{a}-\sqrt{b}}{\sqrt{2(a+b)}}
\]
which, solved for $a/b$ gives $a/b=3.$ Because both sides of \eqref{quozi} give $\sqrt3$, it follows that
\begin{equation}\label{sqrt3}
{\mathrm{F}_{D}^{(3)}\left( \left. 
\begin{array}{c}
\frac12;\frac12 \\[2mm]
2
\end{array}
\right| \frac12,\frac{1}{4},-\frac{1}{2}\right)}=2 \sqrt{3} \left(2 -\sqrt{3}\right)\,{\mathrm{F}_{D}^{(3)}\left( \left. 
\begin{array}{c}
\frac12;\frac12 \\[2mm]
1
\end{array}
\right| \frac12,\frac{1}{4},-\frac{1}{2}\right)}.
\end{equation} 
While from \eqref{quozi2} one finds:
\begin{equation}\label{sqrt3b}
{\mathrm{F}_{D}^{(3)}\left( \left. 
\begin{array}{c}
\frac32;\frac12 \\[2mm]
2
\end{array}
\right| -1,\frac{1}{3},-\frac{1}{3}\right)}=2\sqrt{3} \left(2 -\sqrt{3}\right)\,{\mathrm{F}_{D}^{(3)}\left( \left. 
\begin{array}{c}
\frac12;\frac12 \\[2mm]
1
\end{array}
\right| -1,\frac{1}{3},-\frac{1}{3}\right)}.
\end{equation} 
In such a way we can find a sequence of values to parameters $a,\,b$ whom singular moduli are corresponding to, as listed in the table at end of the section, by which some identities can be got concerning the functions $\mathrm{F}_{D}^{(3)}$ similar to \eqref{sqrt3} and \eqref{sqrt3b} according to different $n-$orders. To shorten our expressions we define
\[
{\rm H}_{1}(x,y,z)={\mathrm{F}_{D}^{(3)}\left( \left. 
\begin{array}{c}
\frac12;\frac12 \\[2mm]
2
\end{array}
\right| x,y,z\right)},\quad {\rm G}(x,y,z)={\mathrm{F}_{D}^{(3)}\left( \left. 
\begin{array}{c}
\frac12;\frac12 \\[2mm]
1
\end{array}
\right| x,y,z\right)},\quad {\rm H}_{2}(x,y,z)={\mathrm{F}_{D}^{(3)}\left( \left. 
\begin{array}{c}
\frac32;\frac12 \\[2mm]
2
\end{array}
\right| x,y,z\right)}
\]
Then identities of the form
\[
{\rm H}_{1}(x,y,z)=R\,{\rm G}(x,y,z),\quad 
\]
hold, where:
\begin{description}
\item Order $5$:
\[
x=\frac12,\,y=\frac{3-\sqrt{5}}{2},\,z=-\frac{1+\sqrt{5}}{2};\quad R=2\sqrt{\sqrt{5}-2}
\]
\item Order $7$:
\[
x=\frac12,\,y=\frac{7}{16},\,z=-\frac{7}{2}; \quad R=\frac{2 \left(4-\sqrt{7}\right)}{\sqrt{7}}
\]
\item Order $9$:
\[
x=\frac12,\,y=2\sqrt3-3,\,z=-(2\sqrt3+3);\quad R=\frac{\sqrt2}{\sqrt[4]{3}}
\]
\item Order $13$:
\[
x=\frac12,\,y=\frac{19-5 \sqrt{13}}{2},\,z=-\frac{17+5\sqrt{13}}{2};\quad R=\frac{2\sqrt{37 \sqrt{13}-106}}{9}
\]
\item Order $15$:
\[
x=\frac12,\,y=\frac{3}{32} \left(3+\sqrt{5}\right),\,z=-\frac{3}{2} \left(9+4\sqrt{5}\right);\quad R=\frac{2 \left(\sqrt{3}-3 \sqrt{5}\right) \left(\sqrt{5}-4\right)}{3\left(1+\sqrt{15}\right)}
\]
\item Order $25$
\[
x=\frac12,\,y=4 \left(9 \sqrt{5}-20\right),\,z=-4 \left(20+9 \sqrt{5}\right);\quad R=\frac{2}{\sqrt[4]{5}}
\]
\item Order $33$
\[
x=\frac12,\,y=\frac{16}{16+15 \sqrt{3}-3 \sqrt{11}},\,z=\frac{16}{16-15 \sqrt{3}+3
   \sqrt{11}};\quad R=\frac{\sqrt{3} \left(17-\sqrt{33}\right) \sqrt[4]{43-5 \sqrt{33}}}{2^{19/4}}
\]
\end{description}
Eventually we have further identities of the form:
\[
{\rm H}_{2}(x,y,z)=R\,{\rm G}(x,y,z),\quad 
\]
where:
\begin{description}
\item Order $5$:
\[
x= -1,\,y=\frac{\sqrt{5}-1}{2},\,z=-\frac{\sqrt{5}-1}{2};\quad R=2 \sqrt{\sqrt{5}-2}
\]
\item Order $7$:
\[
x= -1,\,y=\frac{7}{9},\,z=-\frac{7}{9};\quad R=\frac{2 \left(4-\sqrt{7}\right)}{\sqrt{7}}
\]
\item Order $9$:
\[
x= -1,\,y=\frac{\sqrt3}{2},\,z=-\frac{\sqrt3}{2};\quad R=\frac{\sqrt2}{\sqrt[4]{3}}
\]
\item Order $13$:
\[
x= -1,\,y=\frac{5 \sqrt{13}-1}{18},\,z=-\frac{5 \sqrt{13}-1}{18};\quad R=\frac{2\sqrt{37 \sqrt{13}-106}}{9}
\]
\item Order $15$:
\[
x= -1,\,y=\frac{3}{121} \left(21+8 \sqrt{5}\right),\,z=-\frac{3}{121}
   \left(21+8 \sqrt{5}\right);\quad R=\frac{2 \left(\sqrt{3}-3 \sqrt{5}\right) \left(\sqrt{5}-4\right)}{3\left(1+\sqrt{15}\right)}
\]
\item Order $25$:
\[
x= -1,\,y=\frac{4\sqrt5}{9},\,z=-\frac{4\sqrt5}{9};\quad R=\frac{2}{\sqrt[4]{5}}
\]
\item Order $33$:
\[
x= -1,\,y=\frac{16}{15 \sqrt{3}-3 \sqrt{11}},\,z=-\frac{16}{15 \sqrt{3}-3 \sqrt{11}};\quad R=\frac{\sqrt{3} \left(17-\sqrt{33}\right) \sqrt[4]{43-5 \sqrt{33}}}{2^{19/4}}
\]
\end{description}
Let us provide a table with some values of the ratio $a/b$ which singular moduli are corresponding to:
\begin{table}[H]
\centering
\begin{tabular}{| l | l || l |}
  \hline                        
  $n$ & $\lambda^\ast(n)$ & $a/b$
  \\
  \hline 
%  2 & $\sqrt2-1$ & $\frac{1}{49} \left(65+80 \sqrt{2}+4 \sqrt{914+650 \sqrt{2}}\right)$
%  \\
%  \hline 
 % \hline     
  3 & $\frac{\sqrt{3}-1}{2 \sqrt{2}}$ & 3 \\
   \hline
%   4 & $3-2\sqrt2$ & $\frac{1}{441} \left(57+352 \sqrt{2}+8 \sqrt{884+627 \sqrt{2}}\right)$ \\
%    \hline  
     
  5 & $\frac{1}{2} \left(\sqrt{\sqrt{5}-1}-\sqrt{3-\sqrt{5}}\right)$ & $\frac{1}{2} \left(1+\sqrt{5}\right)$ \\
  \hline  
  7 & $\frac{3-\sqrt{7}}{4 \sqrt{2}}$ & $9/7$ \\
      \hline
   9 & $\frac{1}{2} \left(\sqrt{2}-\sqrt[4]{3}\right) \left(\sqrt{3}-1\right)$ & $2/\sqrt3$ \\
   \hline 
    13 & $\frac{1}{2} \left(\sqrt{5 \sqrt{13}-17}-\sqrt{19-5 \sqrt{13}}\right)$ & $\frac{1}{18} \left(1+5 \sqrt{13}\right)$ \\
   \hline 
    15 & $\frac{\left(2-\sqrt{3}\right) \left(3-\sqrt{5}\right) \left(\sqrt{5}-\sqrt{3}\right)}{8
   \sqrt{2}}$ & $\frac{1}{3} \left(21-8 \sqrt{5}\right)$ \\
 %  \hline 
  % 21 & $\lambda^\ast(21)$ & $\frac{1}{6} \left(7-4 \sqrt{3}+4 \sqrt{7}-\sqrt{21}\right)$\\
   \hline 
   25 & $\frac{\left(3-2 \sqrt[4]{5}\right) \left(\sqrt{5}-2\right)}{\sqrt{2}}$ & $\frac{9}{4 \sqrt{5}}$\\
   \hline
   33 & $\lambda^\ast(33)$ & $\tfrac{3}{16} \left(5 \sqrt{3}-\sqrt{11}\right)$\\
%   \hline
 %   37 & $\frac{1}{2} \left(\sqrt{145 \sqrt{37}-881}-\sqrt{883-145 \sqrt{37}}\right)$ & $\frac{1}{882} \left(1+145 \sqrt{37}\right)$\\
   \hline
   
\end{tabular}
%\caption{Singular moduli and ratio $a/b$}\label{bella}

\[
\lambda^\ast(33)=\frac{1}{2} \left(\sqrt{261-150 \sqrt{3}-78 \sqrt{11}+45
   \sqrt{33}}-\sqrt{-259+150 \sqrt{3}+78 \sqrt{11}-45 \sqrt{33}}\right)
 %  \end{split}
\]

\end{table}
%\hspace{0.5cm}

\section*{Conclusions}
By extending a Legendre approach to the reduction of hyperelliptic integrals, \cite{leg, leg2} and founding upon the so called Cauchy-Schl{\"o}milch  transformation, we obtained six formulae for $\pi$ to be added to those we previously proved  in \cite{jnt1, jnt2}. Some evaluations of Lauricella functions in their analytic continuation, Theorem \ref{prolo}, formulae \eqref{continuation}, have been given.
Because some complete elliptic integrals $\mathbf{K}_{+}^{(a,b)},\,  \mathbf{K}_{-}^{(a,b)},$ introduced in equation \eqref{mentari}, and obtained by the above transformation, have complementary modules, some values are found to parameters of those families of complete elliptic integrals providing singular moduli. Therefore by means of singular moduli theory we highlighted some identities concerning Lauricella functions type ${\rm F}_D^{(3)}$  for some special values of their arguments and parameters. 

Even if our initial formulae have been obtained through elementary methods, nevertheless through the non-elementary singular moduli theory new results are allowed both on the Lauricella’s functions and about $\pi.$ as well. 
The exact meaning of the adjective {\it elementary} in a mathematical context is well explained in the book of J. Havil, \cite{havil2009gamma}, where, at pp. xx-xxi of the introduction he states:
\begin{quote}
Mathematics makes a nice distinction between the usually synonymous terms \lq\lq elementary'' and \lq\lq simple'', with \lq\lq elementary'' taken to mean that not very much mathematical knowledge is needed to read the work and \lq\lq simple'' to mean that not very much mathematical ability is needed to understand it.
\end{quote}
and whose opinion we share completely.

%\bibliographystyle{elsarticle-num}
%\bibliography{JNT3}

\end{document}